\documentclass[11pt]{article}
\usepackage{amsmath, latexsym, amsfonts, amssymb, amsthm, amscd}
\usepackage{xcolor}
\usepackage[hidelinks]{hyperref}
\usepackage{graphicx}
\usepackage{placeins}
\DeclareGraphicsExtensions{.pdf,.png,.jpg,.eps}
\newtheorem{theorem}{Theorem}
\newtheorem{corollary}{Corollary}
\newtheorem{proposition}{Proposition}
\newtheorem{lemma}{Lemma}

\title{
\textbf{Generalized scale functions for spectrally negative L\'evy processes}}

\author{\textbf{ Jes\'us Contreras\footnote{Centro de Investigaci\'on en Matem\'aticas A.C. Calle Jalisco s/n. C.P. 36240, {\sc Guanajuato, Mexico.}  Email: jjcontreras@cimat.mx.}\,\,,  V{\'i}ctor Rivero\footnote{Centro de Investigaci\'on en Matem\'aticas A.C. Calle Jalisco s/n. C.P. 36240, {\sc Guanajuato, Mexico.}  Email: rivero@cimat.mx.}\, .}}
\date{\footnotesize This version: \today}

\begin{document}
\maketitle

\begin{abstract}
\bigskip

For a spectrally negative L\'evy process, scale functions appear in the solution of two-sided exit problems, and in particular in relation with the Laplace transform of the first time it exits a closed interval. In this paper, we consider the Laplace transform of more general functionals, which can depend simultaneously on the values of the process and its supremum up to the exit time. These quantities will be expressed in terms of generalized scale functions, which can be defined using excursion theory. In the case the functional does not depend on the supremum, these scale functions coincide with the ones found on the literature, see e.g. \cite{li2018fluctuations} and therefore the results in this work are an extension of them.
\bigskip

\noindent {\sc Key words}:  Scale functions, Local times,  L\'evy processes, Fluctuation theory.\\
\noindent MSC 2010 subject classifications: 60J99, 60G51.
\end{abstract}

\vspace{0.5cm}

\section{Introduction}

Let $X=(X_t,t\geq 0)$ be a real-valued spectrally negative Lévy process (SNLP for short), that is, a stochastic process with independent and stationary increments and no positive jumps. These properties allow to define the Laplace exponent $\Psi$ of $X$, which is given by $$\mathbb{E}\left[e^{\lambda X_t} \right]= e^{t \Psi(\lambda)}, \ \ \ \lambda,t \geq 0,$$ and by the L\'evy-Khintchine formula, $$\Psi(\lambda)=a\lambda+\frac{\sigma^2}{2}\lambda^2 + \int_{(-\infty,0)} \Pi(dx) \left(e^{\lambda x}-1-\lambda x 1_{\{|x|<1\}} \right),$$ where $a,\sigma \in \mathbb{R}$ and $\Pi$ is a $\sigma-$finite measure over $(-\infty,0)$ satisfying the condition $$\int_{(-\infty,0)} \Pi (dx) (1\wedge x^2)  <\infty.$$ Actually, $\Psi$ characterizes $X$ and codes various informations about it.  

Indeed, the right derivative of $\Psi$ at 0 provides the asymptotic behavior of $X$:
\begin{itemize}
\item[i)] $\Psi'(0+)>0$ if and only if $X_t\to \infty$, a.s.,
\item[ii)] $\Psi'(0+)<0$ if and only if $X_t\to -\infty$, a.s. and
\item[iii)] $\Psi'(0+)=0$ if and only if $\limsup_{t\to \infty} X_t= \infty$ and $\liminf_{t\to \infty} X_t= -\infty$, a.s.
\end{itemize}

Since $\Psi$ is strictly convex and $\Psi(0)=0$, either $0$ is the only root of the equation $\Psi(\lambda)=0,$ or there is an additional root. Denote by $\Phi(0)$ the largest non-negative root. From this value on, $\Psi$ is continuous and strictly increasing and so we can define its inverse $\Phi:[0,\infty)\to [\Phi(0),\infty)$.

\underline{\textit{Assumption:}} We will assume throughout this paper that $X$ has paths of unbounded variation, which for the spectrally negative L\'evy process is equivalent to require either $\sigma\neq 0$ or $\int_{-1}^0 \Pi(dx) |x|=\infty$.

Our main goal is to generalize the classical scale functions for $X$ with an excursion theory approach. For $q\geq 0$, denote by $W^{(q)}$ and $Z^{(q)}$, the usual $q-$scale functions for $X$. The function $W^{(q)}:\mathbb{R}\to [0,\infty)$ is such that $W^{(q)}(x)=0$ for $x<0$ and over $[0,\infty)$ it is the only strictly increasing and continuous function whose Laplace transform is given by $$\int_0^\infty dx e^{-\beta x} W^{(q)}(x) = \frac{1}{\Psi(\beta)-q}, \ \ \ \beta>\Phi(q).$$ On the other hand, $Z^{(q)}$ is defined by $$Z^{(q)}(x)=1+q\int_0^x dy W^{(q)}(y), \ \ \ x\in \mathbb{R}.$$ We simply write $W$ and $Z$ for the scale functions associated to $q=0$.  These functions, its extensions and applications are constantly applied in recent research projects, and it becomes imposible nowadays to give a panorama of the topic. Nevertheless, for a deeper insight about these functions we refer to \cite[Ch. VIII]{kyprianou2014fluctuations} and \cite{kuznetsov2012theory}, where an overview of the topic can be found.

Denote by $S_t$ the current supremum of $X$ up to time $t>0$. Scale functions first appeared while studying the exit problem of an interval for $X$. Before stating that problem, we introduce the first passage times above and below a given level $c\in \mathbb{R}$: $$\tau_c^+=\inf \{t>0:X_t >c \}, \ \ \ \tau_c^-=\inf \{t>0: X_t<c \}.$$ For $a,b\in \mathbb{R}$, $b<a$ and $x\in [b,a]$, the previous functions allow to express the Laplace transform of $T=\tau_a^+ \wedge \tau_b^-$, which is the first time the process leaves the interval $[b,a]$, in a closed form, namely
\begin{equation}\label{qscale1}
\mathbb{E}_x \left[e^{-qT},T=\tau_a^+ \right]=\frac{W^{(q)}(x-b)}{W^{(q)}(a-b)}
\end{equation}
and
\begin{equation}\label{qscale2}
\mathbb{E}_x \left[e^{-qT},T=\tau_b^- \right]=Z^{(q)}(x-b)-Z^{(q)}(a-b)\frac{W^{(q)}(x-b)}{W^{(q)}(a-b)}.
\end{equation}
In particular, when $q=0$, these expressions give the probabilities that $X$ exits the interval $[b,a]$ by above or below, respectively.  

Besides, scale functions can be described and analyzed using excursion theory. It is well known (see for instance \cite[Ch. VIII.2]{kyprianou2014fluctuations}) that the probability of leaving the interval by above can be expressed as
\begin{equation}\label{equation1}
\mathbb{P}_x \left(\tau_a^+ < \tau_b^- \right)=\frac{W(x-b)}{W(a-b)}=\exp\left\{-\int^{a-b}_{x-b} ds \overline{N}\left( H>s\right) \right\}=\exp\left\{-\int^{a}_x ds \overline{N}\left( H>s-b\right) \right\},
\end{equation}
where $\overline{N}$ is the so called \textit{excursion measure} of the reflected process $S-X$. For a generic excursion $\mathbf{e}$, $\zeta=\zeta(\mathbf{e})$ denotes its lifetime or duration and $H=H(\mathbf{e})$ is its height. Further details about excursions will be described in the forthcoming Section \ref{Preliminaries}.

Now, consider the more general functional of the path of $X,$ stopped at time $T,$ given by $\int_0^T dt f(X_t)$, for a non-negative and locally bounded function $f$. In \cite{li2018fluctuations}, Li and Palmowski found a nice expression for its Laplace transform but in terms of what they called \textit{generalized} scale functions, $W^{(f)}$ and $Z^{(f)}$. Their Corollary 2.2 states that
$$\mathbb{E}_x\left[\exp\left\{-\int^{T}_0 dtf(X_t)\right\} , T=\tau^{+}_{a}\right]=\frac{W^{(f)} (x,b)}{W^{(f)}(a,b)}$$ and $$\mathbb{E}_x \left[  \exp \left\{ - \int_0^{T} dt f(X_t) \right\}, T=\tau_{b}^-  \right] =Z^{(f)}(x,b)-\frac{W^{(f)}(x,b)}{W^{(f)}(a,b)}Z^{(f)}(a,b).$$ In this setting, the generalized scale functions are no longer characterized as in the classical case. They are determined by being the only locally bounded solutions to the renewal like equations $$W^{(f)}(u,v)=W(u-v) + \int_v^u dz W(u-z)f(z)W^{(f)}(z,v),$$ and $$Z^{(f)}(u,v)=1+\int_{v}^u dz W(u-z) f(z)Z^{(f)}(z,v),$$ where $\int_v^u =0$ if $v>u$. Notice that one can recover the Laplace transform of $T$ by considering $f\equiv q$.

Here, we will continue the work in \cite{li2018fluctuations} in the sense of describing generalized scale functions using excursion theory. This will allow us to consider functionals depending on both $X$ and $S$, namely,  $$\int_0^T dt f(S_t, X_t),$$ where $f:\mathbb{R}^2 \to [0,\infty)$ is measurable and bounded. Also, this will lead us to describe the asymptotic behavior of generalized scale functions, to extend the above quoted identities to these functionals and to establish a new results. In particular, we will derive information on the law of this functional conditionally on the supremum attained prior to the first exit from the interval $[b,a]$ and information on the last position of the process before this time.   

\section{Main results}

For a function $f:\mathbb{R}^2 \to \mathbb{R}_+$ measurable and bounded, we will call {\it generalized scale function} $W_f:\mathbb{R}^2 \to \mathbb{R}$ to the function defined by $W_f(x,b)=0,$ if $x\leq b,$ and, for $x>b,$
\begin{align} \label{DefWfN}
W_f(x,b):=& W(x-b)\exp\left\{ \int_0^{x-b} ds \overline{N} \left(1-\exp \left\{-\int_0^\zeta du f(s+b,s+b-\mathbf{e}(u)) \right\}, H<s \right)   \right\} \\
= & W(x-b) \exp\left\{ \int_b^{x} ds \overline{N} \left(1-\exp \left\{-\int_0^\zeta duf(s,s-\mathbf{e}(u)) \right\}, H<s-b \right) \right\}. \nonumber
\end{align} Where $W$ denotes the classical $0$-scale function associated to $X.$

Considering additional measurable and bounded functions $g:\mathbb{R} \to \mathbb{R}$ and $h:\mathbb{R}^2 \to \mathbb{R}$, we define the {\it generalized $Z$-scale function}, $Z_{f,g,h}:\mathbb{R}^2 \to \mathbb{R}$ as $Z_{f,g,h}(x,b)=1,$ for $x\leq b,$ while for $x>b,$ we take
\begin{equation} \label{DefZfg}
\begin{split}
Z_{f,g,h}(x,b) &:= W_f(x,b) \left( 1 + \int_{x-b}^{\infty} dz g(z+b) \frac{\kappa_b(z+b;f,h)}{W_f(z+b,b)} \right)\\
&= W_f(x,b) \left( 1 + \int_{x}^{\infty} dz g(z) \frac{\kappa_b(z;f,h)}{W_f(z,b)} \right),
\end{split}
\end{equation}
where for $z>b$, $$\kappa_b(z;f,h):=\overline{N} \left( h\left(z-\mathbf{e}\left(\tilde{\tau}_{z-b}^+-\right),z-\mathbf{e}\left(\tilde{\tau}_{z-b}^+ \right) \right)\exp \left\{-\int_0^{\tilde{\tau}_{z-b}^+} du f(z,z-\mathbf{e}(u)) \right\},H>z-b \right)$$ and $\tilde{\tau}^+_c$ the first time a generic excursion $\mathbf{e}$ surpasses level $c$. Notice that the above defined functions are defined directly in terms of scale functions and excursion measure, not recursively. 

We can now state our first main theorem.

\begin{theorem} \label{Th1}
Let $X$ be a spectrally negative L\'evy process with unbounded variation, Laplace exponent $\Psi$ and $S$ its current supremum. For any $a,b \in \mathbb{R}$ such that $b<a$, $x\in (b,a)$, $f,h:\mathbb{R}^2 \to [0,\infty)$ and $g:\mathbb{R} \to \mathbb{R}$ measurable and bounded, if $T=\tau_a^+ \wedge \tau_b^-$ then

\begin{equation} \label{exitWf}
\mathbb{E}_x\left(\exp\left\{-\int^{T}_0 dtf(S_t, X_t)\right\} , T=\tau^{+}_{a}\right)=\frac{W_f (x,b)}{W_f(a,b)},
\end{equation}

and

\begin{equation} \label{exitZf}
\begin{split}
&\mathbb{E}_x \left[ g\left(S_{\tau_{b}^-} \right)h\left(X_{\tau_b^--},X_{\tau_b^-} \right) \exp \left\{ - \int_0^{T} dtf(S_t,X_t)  \right\}, T=\tau_{b}^- \right] \\
&=Z_{f,g,h}(x,b)-\frac{W_f(x,b)}{W_f(a,b)}Z_{f,g,h}(a,b).\\
\end{split}
\end{equation}
\end{theorem}

This result preserves the structure for the expressions of the Laplace transform and extends that of Li and Palmowski, which can be recovered by considering $f$ depending only on the second entry and $g\equiv h \equiv 1$. The functions $W_f$ and $Z_{f,g,h}$ are expressed in terms of $\overline{N}$, as an alternative to the issue of finding solutions to integral equations. There are various descriptions of the excursion measure $\overline{N}$ that can be useful here, and we refer to \cite[Ch. 7]{bertoin1996levy} and \cite[Ch. 8]{kyprianou2014fluctuations}.   

The definition of $W_f$ implies the following corollary.
\begin{corollary} \label{Cor1}
For every $b\in \mathbb{R}$, the asymptotic behavior of the scale function $W_f(\cdot,b)$ is determined by
\begin{equation}
\exp \left\{-\int_b^x ds  \overline{N} \left(1-\exp \left\{-\int_0^\zeta du f(s,s-\mathbf{e}(u)) \right\},H<s-b \right)  \right\} \frac{W_f(x,b)}{W(x-b)} \to 1,
\end{equation}
as $x\to \infty$.
\end{corollary}

\textbf{Remarks. 1)} This approach provides an expression for functionals that can depend on both $X$ and its supremum. For instance, taking $f(u,v)$ depending only on $u-v$ gives information about the reflected process $S-X$. Additionally, if $f\equiv 0$ and $h\equiv 1$, equation (\ref{exitZf}) provides an expression for the last position of the supremum prior to the exit by below, which has not been found on the literature before, namely $$\mathbb{E}_x \left[g\left(S_{\tau_b^-} \right),T=\tau_b^- \right]=\int_x^a dz g(z) \frac{W(x-b)}{W(z-b)} \overline{N} \left(H>z-b \right).$$ Similarly, considering $f\equiv 0$ we obtain information on the points visited by $X$ prior and after the exit by $b$, and the supremum:\
\begin{align*}
&\mathbb{E}_x \left[g\left(S_{\tau_b^-} \right)h\left(X_{\tau_b^--},X_{\tau_b^-} \right),T=\tau_b^- \right]\\
&=\int_x^a dz \overline{N}(H>z-b)g(z)\frac{W(x-b)}{W(z-b)} \frac{\overline{N} \left(h\left(z-\mathbf{e}\left(\tilde{\tau}_{z-b}^+- \right),z-\mathbf{e}\left(\tilde{\tau}_{z-b}^+ \right) \right),H>z-b \right)}{\overline{N}(H>z-b)}.
\end{align*}

\textbf{2)} Observe that one can change the definition of $W_f(\cdot,b)$ by multiplying by a constant $c>0$ or consider a constant $d>0$ instead of 1 in the definition of $Z_{f,g,h}$ and the result of Theorem \ref{Th1} remains true.

\textbf{3)} In the forthcoming Theorem \ref{Th3} we prove that in case $g,h\equiv 1$ and $f$ depends on $X$ and not in the supremum $S$, the scale functions $W_f$ and $Z_f$ coincide with those on \cite{li2018fluctuations} up to constants, as in the previous remark. To provide a particular example, take $f\equiv 0$ and $b=0$. We have by definition that $W_0(x,0)=W(x)$ and for $Z_0(x,0)$,
\begin{equation*}
Z_0(x,0)= W_0(x,0) \left(1+\int_x^\infty dz \frac{\overline{N}(H>z)}{W_0(z,0)} \right) = W(x) \left(1+\int_x^\infty dz \frac{\overline{N}(H>z)}{W(z)} \right).
\end{equation*}
From Lemma 8.2 in \cite{kyprianou2014fluctuations}, since $X$ is of unbounded variation, $W$ is differentiable and $W'(z)/W(z)=\overline{N}(H>z)$. Therefore,
\begin{eqnarray*}
Z_0(x,0)&=& W(x) \left(1-\int_x^\infty dz \frac{d}{dz} \left(\frac1{W(z)} \right) \right) \\
&=& W(x) \left(1-\frac1{W(\infty)} + \frac1{W(x)} \right) \\
&=& 1+\left(1-\frac1{W(\infty)} \right)W(x).
\end{eqnarray*}
Since $W$ is a non-decreasing function, either $1/W(\infty)$ is zero or a positive constant. In this case, changing the 1 in the definition for $d=1/W(\infty)$, we have that $Z_0(x,0)=1=Z(x)$, that is, the generalized scale function $Z$ coincides with the classical one up to that constant $d$. Using a change of measure, with similar arguments one can verify that $q$-scale functions satisfy the definitions in (\ref{DefWfN}) and (\ref{DefZfg}). 

\vspace{0.3cm}

Furthermore, in the forthcomming Proposition~\ref{Pr1} it is proved that we can completely express $W_f(\cdot,b)$ in terms of $\overline{N}$. See Section \ref{proofs}.

There is also a relation between the equations on Theorem \ref{Th1} and the theory of local times. Recall that $X$ has an associated process $(L_t^x,t\geq 0, x\in \mathbb{R})$ where $L_t^x$ is the \textit{local time of $X$ at $x$ up to time $t$} and can be interpreted as the amount of time the process stays in state $x$ during $[0,t]$. Local times satisfy the \textit{occupation formula}:
\begin{equation} \label{OcFor}
\int_0^\tau dt f(X_t) \stackrel{\text{a.s.}}{=} \int_\mathbb{R} dy f(y)L_\tau^y,
\end{equation}
for any $f\geq 0$ measurable and bounded and stopping times $\tau$. Therefore, we have the following corollary.

\begin{corollary}
For $x\in (b,a)$, the Laplace transform of $\int_b^a dy f(y)L_T^y$ under $\mathbb{P}_x$ is given by 
\begin{align}
&\mathbb{E}_x \left[\exp\left\{-\int_b^a dy f(y)L_T^y \right\} \right]=Z_f(x,b)-\frac{W_f(x,b)}{W_f(a,b)} \left(Z_f(a,b)-1 \right) \\ \nonumber
& \hspace{2.5cm}= \frac{W_f(x,b)}{W_f(a,b)} + \int_x^a dz \frac{W_f(x,b)}{W_f(z,b)} \overline{N} \left(\exp\left\{-\int_b^{z} dv L^{z-v}_{\tau^{+}_{z-b}}(\mathbf{e})f(v) \right\},H>z-b \right).
\end{align}
\end{corollary}

This result is obtained just by adding (\ref{exitWf}) and (\ref{exitZf}) and it can help in characterizing the finite dimensional distributions of the local time process $\left(L_T^y, y\in [b,a] \right)$, but it is a problem that we do not explore here. Regarding this matter, in \cite{li2020local} Li and Zhou provide an expression for the joint Laplace transform of the local times at different levels $x_1,\dots,x_n \in (b,a)$. Thus this is an improvement of their result.

In the next theorem we use the arguments in the proof of identity (\ref{exitZf}) to describe the Laplace transform of $\int_0^{T} dt f(S_t,X_t)$ conditionally on the height of the supremum at the first exit time. 
\begin{theorem} \label{Th2}
The conditional expectation of the Laplace transform of the functional $\int_0^T dt f(S_t,X_t)$ given the supremum reached up to time $T$ is a.s. given by

\begin{align} \label{eqTh2}
&\mathbb{E}_x \left[ \exp\left\{ -\int_0^{T} dt f(S_t,X_t) \right\} \bigg| S_{T}=z \right]  \nonumber \\
& \ \ \ \ \ \ \  = \frac{W(z-b)}{W(x-b)} \frac{W_f(x,b)}{W_f(z,b)} \times \begin{cases}
\frac{\overline{N} \left(\exp \left\{-\int_0^{\tilde{\tau}^+_{z-b}}du f(z,z-\mathbf{e}(u)) \right\} , H>z-b\right)}{\overline{N}(H>z-b)}, & \ \ \text{if } x< z<a, \\
1, & \ \ \text{if }z=a.
\end{cases}
\end{align}
\end{theorem}

\textbf{Remark 4)} Using Theorem \ref{Th1} and equation (\ref{equation1}), the first two quotients in the right hand side of (\ref{eqTh2}) are precisely $$\mathbb{E}_x \left[ \exp\left\{ -\int_0^{\tau_z^+} dt f(S_t,X_t) \right\} \bigg| \tau_z^+ <\tau_b^- \right].$$ Hence, conditional to $S_T=z$, the process first needs to reach $z$ and this term represents the contribution of the functional on this event. And if $z \in (x,a)$, the second term represents the contribution of the last excursion from the supremum at level $z$ that makes the process leave the interval by below. 

%
%
%
%
Finally, the next theorem relates the scale functions $W_f$ and $Z_f$ defined before to the scale functions $W^{(f)}$ and $Z^{(f)}$ in \cite{li2018fluctuations}.

\begin{theorem} \label{Th3}
If $f(u,v)\equiv f(v)$, i.e. $f$ depends only on the second entry, then the generalized scale functions $W_f(\cdot,b)$, $Z_f(\cdot,b)$ and $W^{(f)}(\cdot,b)$, $Z^{(f)}(\cdot,b)$ differ by positive constants in the sense of Remark 2).
\end{theorem}

\section{Preliminaries }\label{Preliminaries}

First, we introduce the space $\mathcal{E}$ of positive right continuous paths with left limits and defined on an interval: $$\mathcal{E}=\left\{\mathbf{e}:[0,\zeta] \to [0,\infty)  \ | \ \zeta \in (0,\infty], \ \mathbf{e}((0,\zeta))\subset (0,\infty) \text{ and }  \mathbf{e} \text{ is c\`adl\`ag}  \right\}.$$ For an element $\mathbf{e}\in \mathcal{E}$, the right endpoint of its interval of definition is called \textbf{duration} or \textbf{lifetime} and it is denoted by $\zeta(\mathbf{e})$. The supremum of $\mathbf{e}$ is called \textbf{height} and is denoted by $H(\mathbf{e})=\sup_{v\in [0,\zeta]} \mathbf{e}(v)$.

For a SNLP $X$, it is well known that one can take $S$ as the local time at the supremum and that its right continuous inverse is a subordinator $(\tau_t^+,t\geq 0)$. For $t>0,$ such that $\tau_t^+ \neq \tau_{t-}^+ := \lim_{s\uparrow t}\tau_s^+$, the supremum is constant and equal to $t$ on the interval $[\tau_{t-}^+,\tau_t^+]$. Therefore, we can define $$\mathbf{e}_t(v):=(S-X)_{\tau_{t-}^+ +v}, \ \ \ 0\leq v\leq \tau_t^+ - \tau_{t-}^+,$$ the excursion of $(S-X)$ at local time $t$. In this case, $\mathbf{e}_t \in \mathcal{E}$ and actually, $\zeta(\mathbf{e}_t)=\tau_t^+ - \tau_{t-}^+$. In case $\tau_t^+=\tau_{t-}^+$, one assigns $\mathbf{e}_t=\delta$, where $\delta \notin \mathcal{E}$ is an auxiliary state. More information on excursion theory can be found on \cite{blumenthal2012excursions} and \cite{fitzsimmons2006excursion}.

A result of excursion theory (see for example \cite[Th. 6.14]{kyprianou2014fluctuations}) states that there exists a measure space $(\mathcal{E},\Sigma,\overline{N})$ such that $\Sigma$ contains the sets of the form $$\left\{\mathbf{e}\in \mathcal{E} : \zeta(\mathbf{e})\in A, H(\mathbf{e})\in B, \mathbf{e}(\zeta)\in C \right\},$$ where $A,B,C$ are Borel sets on $\mathbb{R}$. Furthermore, if $\limsup_{t \to \infty}X_t=\infty$ a.s. (that is, iff $\Psi'(0+)\geq 0$), then $\{(t,\mathbf{e}_t) : t>0, \mathbf{e}_t \neq \delta \}$ is a Poisson point process of intensity $ds\otimes \overline{N} (d\mathbf{e})$. 

On the other hand, if $\limsup_{t\to \infty} X_t < \infty$ a.s. (that is, iff $\Psi'(0+)<0$), then $\{ (t,\mathbf{e}_t): 0<t \leq S_\infty,\mathbf{e}_t \neq \delta \}$ is a Poisson point process of intensity $ds \otimes \overline{N}(d\mathbf{e})$ stopped at the first arrival of an excursion in $\mathcal{E}_\infty := \{\mathbf{e}\in \mathcal{E}:\zeta(\mathbf{e})=\infty \}$, the set of paths with infinite lifetime. For a deeper insight on the excursion measure, we refer to \cite{chaumont2005levy} and \cite{duquesne2003path}.

Hence, one can use the tools from the theory of Poisson point processes, such as the compensation and exponential formulas, to perform computations related to excursions of $X$.

Under the assumption of unbounded variation for $X$, Theorem 6.7 from \cite{kyprianou2014fluctuations} implies that $$\int_0^x dt1_{\{S_t=X_t \}} =0, \ \forall x\geq 0.$$ Therefore, the set of points $\{t\geq 0: S_t=X_t \}$ on which $X$ attains a new supremum has Lebesgue measure 0, a.s. If $f$ is measurable and locally bounded, this allows us to decompose the functional $\int_0^{\tau_a^+} dt f(S_t,X_t)$ into
\begin{eqnarray*}
\int_0^{\tau_a^+} dt f(S_t,X_t)&=& \int_{\{t\leq \tau_a^+ : S_t=X_t \}} dt f(S_t,X_t) + \int_{\{t\leq \tau_a^+ : S_t>X_t \}} dt f(S_t,X_t)  \\
&\stackrel{\text{a.s.}}{=}& \int_{\{t\leq \tau_a^+ : S_t>X_t \}} dt f(S_t,X_t).
\end{eqnarray*}
Since the excursions from the supremum occur on this latter set, we have
\begin{eqnarray*}
\int_0^{\tau_a^+} dt f(S_t,X_t)& \stackrel{\text{a.s.}}{=}& \sum_{0<v \leq a} \int_{\tau^+ _{v-}}^{\tau_v^+} dt f(S_t,X_t) \\
&=& \sum_{0<v \leq a} \int_{0}^{\zeta(\mathbf{e}_v)} du f(v,v-\mathbf{e}_v(u)).
\end{eqnarray*}
We will use this decomposition in the proofs of the results.

Finally, we explore the change in intensity of the excursion process under a change of measure for $X$. It is well known that if $X$ is a SNLP and $q\geq 0$, the process $$M_t(q):= e^{\Phi(q) X_t -qt}, \ \ \ t\geq 0,$$ is a martingale and it has an associated change of measure, called Esscher transform, given by
\begin{equation}
\frac{d\mathbb{P}^{\Phi(q)}}{d\mathbb{P}} \bigg|_{\mathcal{F}_t} =M_t(q), \ \ \ t\geq 0.
\end{equation}
This definition is also valid if one replaces $t$ by a stopping time $\tau$, on the event $\{\tau<\infty \}$.

Under $\mathbb{P}^{\Phi(q)}$, $X$ remains in the class of SNLPs with unbounded variation (see \cite[Ch. 8.1]{kyprianou2014fluctuations}) but its Laplace exponent is $$\Psi_q(\lambda)= \Psi(\lambda + \Phi(q)) -q, \qquad \lambda \geq 0.$$ In particular, since $\Psi$ is strictly convex and $\Phi(0)$ is its largest root, $$\Psi'_q(0+)=\Psi'(\Phi(q)+)\geq \Psi'(\Phi(0)+) \geq 0.$$ Therefore, $X_t \to \infty,$ $\mathbb{P}^{\Phi(q)}-$a.s. if $\Phi(q)>0$ and if $q=0$ and $\Phi(0)=0$, at least one has that $\limsup_{t\to \infty}X_t = \infty$. This change of measure will be useful for the upcoming proofs, along with the following lemma.

\begin{lemma} \label{LmPhiq}
Let $q\geq 0$ and $X$ a SNLP of unbounded variation, with $\overline{N}$ its excursion measure from the supremum. Then, under $\mathbb{P}^{\Phi(q)}$ the process $(t,\mathbf{e}_t)_{t>0}$ is a Poisson point process with intensity $ds\otimes \overline{N}\left(e^{-q\zeta}1_{\{\zeta<\infty\}},d\mathbf{e} \right)$. 
\end{lemma} 

\begin{proof}[Proof of Lemma \ref{LmPhiq}.]
For $q>0$, from equation (7) in \cite[p.120]{bertoin1996levy} with the function $q1_{(0,\infty)}$ we have
\begin{eqnarray*}
\overline{N}\left(1-e^{-q\zeta} \right)&=& \Phi(q) \mathbb{E} \left[ \int_0^\infty dt qe^{-qt} 1_{(0,\infty)}(S_t-X_t) \right]\\
&=& \Phi(q) \mathbb{E} \left[ \int_0^\infty dt qe^{-qt} \right] \\
&=& \Phi(q),
\end{eqnarray*} 
where the second equation follows from the fact that increasing times of a SNLP with unbounded variation have Lebesgue measure 0. For $q>0,$ we write $$\overline{N}\left(1-e^{-q\zeta} \right)=\overline{N}\left(\left[1-e^{-q\zeta} \right],\zeta<\infty \right)+ \overline{N}\left(\zeta=\infty \right)$$ and taking the limit as $q\downarrow 0$ we obtain $\overline{N}(\zeta=\infty)=\Phi(0)$.

Now consider $a>0$ and $g:(0,\infty)\times (\mathcal{E}\cup \{\delta \}) \to \mathbb{R}_+$ measurable and bounded (with $g(t,\delta)=0$). Observe that $\exp\left\{ -\sum_{0<t\leq a} g(t,\mathbf{e}_t) \right\} 1_{\{\tau_a^+<\infty\}}$ is $\mathcal{F}_{\tau_a^+}-$measurable and integrable, since those excursions depend on the trajectory of the process up to $\tau_a^+$. Since $\limsup_{t\to \infty} X_t= \infty,$ $\mathbb{P}^{\Phi(q)}-$a.s., then $\tau_a^+<\infty,$ a.s. and therefore $$\mathbb{E}^{\Phi(q)} \left[\exp\left\{ -\sum_{0<t\leq a} g(t,\mathbf{e}_t)\right\} \right] = \mathbb{E}^{\Phi(q)} \left[\exp\left\{ -\sum_{0<t\leq a} g(t,\mathbf{e}_t)\right\} 1_{\{\tau_a^+ <\infty \}} \right].$$ Hence, using the Esscher transform
\begin{eqnarray*}
\mathbb{E}^{\Phi(q)} \left[\exp\left\{ -\sum_{0<t\leq a} g(t,\mathbf{e}_t)\right\} \right] &=& \mathbb{E}^{\Phi(q)} \left[\exp\left\{ -\sum_{0<t\leq a} g(t,\mathbf{e}_t)\right\} 1_{\{\tau_a^+ <\infty \}} \right] \\
&=& \mathbb{E} \left[\exp\left\{\Phi(q)X_{\tau_a^+} -q\tau_a^+ -\sum_{0<t\leq a} g(t,\mathbf{e}_t)\right\} 1_{\{\tau_a^+ <\infty \}} \right] \\
&=& e^{a\Phi(q)} \mathbb{E} \left[\exp\left\{  -\sum_{0<t\leq a} [q\zeta(\mathbf{e}_t)+g(t,\mathbf{e}_t)]1_{\{\zeta(\mathbf{e}_t) < \infty \}}\right\} \right. \\
&\ & \left. \times 1_{\{\#\{(t,\mathbf{e}_t): t\in (0,a],\zeta(\mathbf{e}_t)=\infty \}=0\}} \right].
\end{eqnarray*}

Since the process of excursions of finite lifetime and infinite lifetime are independent, the last line equals $$e^{a\overline{N}\left(1-e^{-q\zeta} \right)} \mathbb{E} \left[\exp\left\{  -\sum_{0<t\leq a} [q\zeta(\mathbf{e}_t)+g(t,\mathbf{e}_t)] 1_{\{\zeta(\mathbf{e}_t)<\infty \}}\right\} \right] \mathbb{P} \left( \#\{(t,\mathbf{e}_t): t\in (0,a],\zeta(\mathbf{e}_t)=\infty \}=0 \right).$$ Using the exponential formula we obtain the further equalities
\begin{align*} 
& = \exp\left\{a\overline{N}\left(1-e^{-q\zeta} \right)\right\} \exp\left\{-\int_0^a ds \overline{N}\left(\left[1-e^{-q\zeta -g(s,\mathbf{e})} \right] 1_{\{\zeta<\infty\}} \right)\right\} e^{-a\overline{N}\left(1_{\{\zeta=\infty)\}} \right)} \\
&= \exp\left\{ a\overline{N}\left(\left[1-e^{-q\zeta}\right]1_{\{\zeta<\infty \}} \right) + a\overline{N}\left(1_{\{\zeta=\infty \}} \right) - a\overline{N} \left(1_{\{\zeta=\infty \}} \right)\right\}\\
&\qquad \times \exp\left\{-\int_0^a ds \overline{N}\left( \left[1-e^{-q\zeta -g(s,\mathbf{e})} \right] 1_{\{\zeta<\infty\}} \right)\right\} \\
&= \exp\left\{-\int_0^a ds \overline{N}\left( \left[1-e^{-q\zeta -g(s,\mathbf{e})} -\left(1-e^{-q\zeta} \right)\right] 1_{\{\zeta<\infty \}} \right)\right\} \\
&= \exp\left\{-\int_0^a ds \overline{N}\left(  e^{-q\zeta} 1_{\{\zeta<\infty \}} \left[1-e^{-g(s,\mathbf{e})} \right] \right)\right\},
\end{align*}
which is exactly the Laplace transform of the functional $g$ of a Poisson point process with intensity $ds\otimes \overline{N}\left(e^{-q\zeta} 1_{\{\zeta<\infty \}} d\mathbf{e} \right)$.

Since this is valid for any $a>0$ and any function $g$ as before, the monotone class theorem implies that under $\mathbb{P}^{\Phi(q)}$, $(t,\mathbf{e}_t)_{t>0}$ is a Poisson point process with intensity $ds \otimes \overline{N}\left(e^{-q\zeta} 1_{\{\zeta<\infty\}} d\mathbf{e} \right)$.
\end{proof}

\section{Proofs}\label{proofs}
We start by proving a Lemma that will prove very useful in establishing Theorems~\ref{Th1} and \ref{Th2}.  
\begin{lemma}\label{Lm1}

\begin{itemize}
\item[i)] The Laplace transform of the functional $\int_0^T f(S_t,X_t)dt$ on the event $\{ \tau_a^+ < \tau_{b}^-\}$ starting from $x\in (b,a)$ is given by
\begin{equation}\label{L1eq1}
\begin{split}
&\mathbb{E}_x\left(\exp\left\{-\int^{T}_0 dt f(S_t, X_t)\right\} , \tau^{+}_{a}<\tau^{+}_{b}\right)\\
&=\frac{W(x-b)}{W(a-b)}\exp\left\{-\int^{a}_{x}ds\overline{N}\left(1-\exp\left\{-\int^{\zeta}_{0} duf(s, s-\mathbf{e}(u))\right\}, H<s-b\right)\right\}.
\end{split}
\end{equation}

\item[ii)]For $g,h: \mathbb{R} \to \mathbb{R}$ measurable and bounded and $x\in (b,a)$,
\begin{equation}\label{L1eq2}
\begin{split}
& \mathbb{E}_x \left[ g\left(S_{\tau_{b}^-} \right)h\left(X_{\tau_{b}^--},X_{\tau_{b}^-} \right) \exp \left\{ - \int_0^{T} dt f(S_t,X_t) \right\}, \tau_{b}^- < \tau_{a}^+ \right] \\
&=\int_x^{a} dzg(z) \frac{W(x-b)}{W(z-b)} \exp \left\{-\int_{x}^z ds \overline{N} \left(1-\exp \left\{ -\int_0^\zeta du f(s,s-\mathbf{e}(u))\right\},H<s-b \right) \right\} \\
&\times  \overline{N} \left(h\left(z-\mathbf{e} \left(\tilde{\tau}_{z-b}^+- \right),z-\mathbf{e} \left(\tilde{\tau}_{z-b}^+ \right) \right)\exp \left\{-\int_0^{\tilde{\tau}^+_{z-b}}du f(z,z-\mathbf{e}(u)) \right\} , H>z-b\right).
\end{split}
\end{equation}

In particular, the Laplace transform of the functional $\int_0^T dt f(S_t,X_t)$ on the event $\{\tau_{b}^-<\tau_a^+ \}$ is obtained taking $g, h\equiv 1$.
\end{itemize}
\end{lemma}
We prove this result by recurring to excursion theory for the process reflected and hence tools from Poisson point processes of excursions. 
\begin{proof}[Proof of Lemma \ref{Lm1}]

Using the decomposition of functionals from the previous section, for the case when $X$ leaves the interval $[b,a]$ from above we have

\begin{equation*}
\begin{split}
&\mathbb{E}_x\left(\exp\left\{-\int^{T}_0 dtf(S_t, X_t)\right\}, \tau^{+}_{a}<\tau^{+}_{b}\right)\\
&= \mathbb{E}\left(\exp\left\{-\int^{\tau^{+}_{a-x}}_0 dt f(x+S_t, x+X_t)\right\}, \tau^{+}_{a-x}<\tau^{-}_{b-x}\right)\\
&=\mathbb{E}\left(\exp\left\{-\sum_{0<s\leq a-x}\int^{\tau^{+}_{s}}_{\tau^{+}_{s-}} dtf(x+S_t, x+X_t)\right\}, \tau^{+}_{a-x}<\tau^{-}_{b-x}\right)\\
&=\mathbb{E}\left(\exp\left\{-\sum_{0<v\leq a-x}\int^{\zeta(\mathbf{e}_v)}_{0} du f\left(x+v, x+v-\mathbf{e}_v(u)\right)1_{\{H(\mathbf{e}_v)<{v-b+x}\}}\right\} 1_A \right),
\end{split}
\end{equation*}
where $A=\{\#\{(v,\mathbf{e}_v): H(\mathbf{e}_v)>v-b+x, 0<v\leq a-x\}=0\}$ indicates that there are no excursions before reaching level $a-x$ whose height is big enough to exit the interval by below, and since $X$ has no positive jumps this coincides with the event $\{\tau_{a-x}^+ < \tau_{b-x}^- \}$.

The point process of excursions can be divided into $$\{(v,\mathbf{e}_v): H(\mathbf{e}_v)>v+b+x, 0<v<\infty\}$$ and $$\{(v,\mathbf{e}_v): H(\mathbf{e}_v)\leq v+b+x, 0<v<\infty\},$$ which are independent Poisson point processes. Thus, by the exponential formula
\begin{equation*}
\begin{split}
&\mathbb{E}_x\left(\exp\left\{-\int^{T }_0 dtf(S_t, X_t)\right\}, \tau^{+}_{a}<\tau^{-}_{b}\right)\\
&=\exp\left\{-\int^{a-x}_0ds\overline{N}\left(1-\exp\left\{-\int^{\zeta}_{0} duf(x+s, x+s-\mathbf{e}(u))\right\}, H<s-b+x\right)\right\}\\
&\qquad \times \exp\left\{-\int^{a-x}_0ds\overline{N}\left( H>s-b+x\right)\right\}\\
&=\exp\left\{-\int^{a}_{x}ds\overline{N}\left(1-\exp\left\{-\int^{\zeta}_{0} du f(s, s-\mathbf{e}(u))\right\}, H<s-b\right)\right\}\\
&\qquad \times \exp\left\{-\int^{a}_{x}ds\overline{N}\left( H>s-b\right)\right\}.
\end{split}
\end{equation*} 
Taking $f\equiv 0$ we recover equation (\ref{equation1}):
\begin{equation*}
\mathbb{P}_x\left(\tau^{+}_{a}<\tau^{-}_{b}\right)=\frac{W(x-b)}{W(a-b)}=\exp\left\{-\int^{a-b}_{x-b}ds\overline{N}\left( H>s\right)\right\}=\exp\left\{-\int^{a}_{x}ds\overline{N}\left( H>s-b\right)\right\}
\end{equation*}
and hence we can express
\begin{equation*}
\begin{split}
&\mathbb{E}_x\left(\exp\left\{-\int^{T}_0 dtf(S_t, X_t)\right\}, \tau^{+}_{a}<\tau^{-}_{b}\right)\\
&=\frac{W(x-b)}{W(a-b)}\exp\left\{-\int^{a}_{x}ds\overline{N}\left(1-\exp\left\{-\int^{\zeta}_{0} du f(s, s-\mathbf{e}(u))\right\}, H<s-b\right)\right\}.
\end{split}
\end{equation*} 

We now continue with the case where $X$ leaves the interval $[b,a]$ by below.  We denote by $g_{\tau^{-}_{b}}$ the last time where $X$ reaches a new supremum before passing below $b$,
$$g_{\tau^{-}_{b}}=\sup\{s<\tau^{-}_{b}: (S-X)_s=0\}.$$ Thus $g_{\tau^{-}_{b}}$ is the left extrema of the excursion from the supremum straddling $\tau^{-}_{b}.$ 
Next, we use that $g_{\tau^{-}_{b}}$ is the left extrema of the first excursion from the supremum whose height, seen from the level of the supremum, is large enough for the path to exit the interval $[b,a].$ This argument is made precise in the following calculation,   
\begin{equation*}
\begin{split}
&\mathbb{E}_x\left(g \left(S_{\tau^{-}_{b}} \right) h\left(X_{\tau_b^--},X_{\tau_b^-} \right) \exp\left\{-\int^{T}_0 dtf(S_t, X_t)\right\}, \tau^{-}_{b}<\tau^{+}_{a}\right)\\
&=\mathbb{E}_x\left(g \left(S_{\tau^{-}_{b}} \right) h\left(X_{\tau_b^--},X_{\tau_b^-} \right) \exp\left\{-\int^{g_{\tau^{-}_{b}}}_0 dt f(S_t, X_t)\right\}\exp\left\{-\int^{\tau^{-}_{b}}_{g_{\tau^{-}_{b}}} dt f(S_t, X_t)\right\}, \tau^{-}_{b}<\tau^{+}_{a}\right)\\
&=\mathbb{E}\left(g \left(x+S_{\tau^{-}_{b-x}} \right) h\left(x+X_{\tau_b^--},x+X_{\tau_b^-} \right)\exp\left\{-\int^{g_{\tau^{-}_{b-x}}}_0 dtf(x+S_t, x+X_t)\right\} \right. \\
& \qquad \left. \times \exp\left\{-\int^{\tau^{-}_{b-x}}_{g_{\tau^{-}_{b-x}}} dtf(x+S_t, x+X_t)\right\} , \tau^{-}_{b-x}<\tau^{+}_{a-x}\right)\\
&=\mathbb{E}\left(\sum_{v\in \mathcal G} g(x+S_v)1_{\{S_v<a-x\}} \exp\left\{-\int^{v}_{0} dt f(x+S_t, x+X_t)\right\}1_{\{\#\{(u,\mathbf{e}_u): H(\mathbf{e}_u)>u-b+x, 0<u<S_v\}=0\}}\right.\\
&\qquad  \times  h\left(x+S_v - \mathbf{e}_{S_v}\left(\tilde{\tau}_{S_v-b+x}- \right),x+S_v - \mathbf{e}_{S_v}\left(\tilde{\tau}_{S_v-b+x} \right) \right)\\
&\qquad  \times \left. \exp\left\{-\int^{\widetilde{\tau}^{+}_{S_v-b+x}}_0 du f(x+S_v, x+S_v-\mathbf{e}_{S_v}(u))\right\}1_{\{H(\mathbf{e}_{S_v})>S_v-b+x\}}\right)
\end{split}
\end{equation*}  
Where $\widetilde{\tau}_{\ell}$ denotes the first passage time above level $\ell$ for the excursion process. 
Applying the compensation formula for Poisson point processes, we obtain
\begin{equation*}
\begin{split}
&\mathbb{E}_x\left(g \left(S_{\tau^{-}_{-b}} \right)\exp\left\{-\int^{T }_0 dt f(S_t, X_t)\right\}, \tau^{-}_{b}<\tau^{+}_{a}\right)\\
&=\mathbb{E}\left(\int^{\infty}_{0} dS_v g(x+S_v)1_{\{S_v<a-x\}}\exp\left\{-\int^{v}_{0} dt f(x+S_t, x+X_t)\right\}1_{\{\#\{(u,\mathbf{e}_u): H(\mathbf{e}_u)>u-b+x, 0<u<S_{v}\}=0\}}\right.\\
&\qquad  \times \overline{N}\left[ \vphantom{\exp\left\{-\int^{\widetilde{\tau}^{+}_{S_v-b+x}}_0 du f(x+S_v, x+S_v-\mathbf{e}(u))\right\}} h\left(x+S_v - \mathbf{e}\left(\tilde{\tau}_{S_v-b+x}- \right),x+S_v - \mathbf{e}\left(\tilde{\tau}_{S_v-b+x} \right) \right) \right. \\
&\qquad \left. \left. \times \exp\left\{-\int^{\widetilde{\tau}^{+}_{S_v-b+x}}_0 du f(x+S_v, x+S_v-\mathbf{e}(u))\right\}1_{\{H(\mathbf{e})>S_v-b+x\}}\right]\right)\\
&=\mathbb{E}\left(\int^{\infty}_{0} dz g(x+z)1_{\{z<a-x\}}\exp\left\{-\int^{\tau^{+}_z}_{0} dtf(x+S_t, x+X_t)\right\}1_{\{\#\{(u,\mathbf{e}_u): H(\mathbf{e}_u)>u-b+x, 0<u<z\}=0\}}\right.\\
&\qquad  \times \overline{N}\left[ \vphantom{\exp\left\{-\int^{\widetilde{\tau}^{+}_{z-b+x}}_0 duf(x+z, x+z-\mathbf{e}(u))\right\}} h\left(x+z - \mathbf{e}\left(\tilde{\tau}_{z-b+x}- \right),x+z - \mathbf{e}\left(\tilde{\tau}_{z-b+x} \right) \right) \right. \\
&\qquad \left. \left. \times \exp\left\{-\int^{\widetilde{\tau}^{+}_{z-b+x}}_0 duf(x+z, x+z-\mathbf{e}(u))\right\}1_{\{H>z-b+x\}}\right]\right).
\end{split}
\end{equation*} 
 
Using Fubini's theorem and proceeding as in the proof of the first part of the Lemma, we derive the desired identity 
\begin{equation*}
\begin{split}
&\mathbb{E}_x\left(g\left(S_{\tau^{-}_{b}} \right)h\left(X_{\tau_b^--},X_{\tau_b^-} \right)\exp\left\{-\int^{T }_0 dtf(S_t, X_t)\right\}, \tau^{-}_{b}<\tau^{+}_{a}\right)\\
&=\int^{a-x}_{0} dz g(x+z)\mathbb{E}\left(\exp\left\{-\int^{\tau^{+}_z}_{0} dt f(x+S_t, x+X_t)\right\}, \tau^{+}_{z}<{\tau^{-}_{b-x}}\right)\\
&  \times  \overline{N}\left[\vphantom{\exp\left\{-\int^{\widetilde{\tau}^{+}_{z-b+x}}_0 duf(x+z, x+z-\mathbf{e}(u))\right\}} h\left(x+z - \mathbf{e}\left(\tilde{\tau}_{z-b+x}- \right),x+z - \mathbf{e}\left(\tilde{\tau}_{z-b+x} \right) \right) \right. \\
& \left. \times \exp\left\{-\int^{\widetilde{\tau}^{+}_{z-b+x}}_0 duf(x+z, x+z-\mathbf{e}(u))\right\},H>z-b+x\right]\\
&= \int^{a-x}_{0} dz g(x+z)\frac{W(x-b)}{W(z+x-b)}\\
& \times\exp\left\{-\int^{z}_{0}ds\overline{N}\left(1-\exp\left\{-\int^{\zeta}_{0} duf(x+s, x+s-\mathbf{e}(u))\right\}, H<s-b+x\right)\right\}\\
& \times \overline{N}\left[ \vphantom{\exp\left\{-\int^{\widetilde{\tau}^{+}_{z-b+x}}_0 duf(x+z, x+z-\mathbf{e}(u))\right\}} h\left(x+z - \mathbf{e}\left(\tilde{\tau}_{z-b+x}- \right),x+z - \mathbf{e}\left(\tilde{\tau}_{z-b+x} \right) \right) \right. \\
& \left. \times \exp\left\{-\int^{\widetilde{\tau}^{+}_{z-b+x}}_0 duf(x+z, x+z-\mathbf{e}(u))\right\},H>z-b+x\right] \\
&=\int^{a}_{x} dz g(z)\frac{W(x-b)}{W(z-b)}\exp\left\{-\int^{z}_{x}ds\overline{N}\left(1-\exp\left\{-\int^{\zeta}_{0} du f(s, s-\mathbf{e}(u))\right\}, H<s-b\right)\right\}\\
& \times \overline{N}\left[h\left(z - \mathbf{e}\left(\tilde{\tau}_{z-b}- \right),z - \mathbf{e}\left(\tilde{\tau}_{z-b} \right) \right)\exp\left\{-\int^{\widetilde{\tau}^{+}_{z-b}}_0 du f(z, z-\mathbf{e}(u))\right\},H>z-b\right]
\end{split}
\end{equation*}
This finishes the proof of the Lemma. 
\end{proof}

\begin{proposition} \label{Pr1}
For $b\in \mathbb{R}$, one can alternatively define $W_f(x,b)=0$ for $x\leq b$ and for $x\in (b,A+b)$,  
\begin{align} \label{DefWf}
W_f(x,b)&= \exp \left\{-\int_{x-b}^{A} ds \overline{N}(\zeta<\infty,H>s) \right\}  \nonumber \\
& \times \exp\left\{ \int_0^{x-b} ds \left[\overline{N} \left(1-\exp \left\{-\int_0^\zeta du f(s+b,s+b-\mathbf{e}(u)) \right\}, H<s \right)+\overline{N}\left(\zeta=\infty \right) \right] \right\} \nonumber \\
&= \exp \left\{-\int_{x}^{A+b} ds \overline{N}(\zeta<\infty,H>s-b) \right\}  \\
& \times \exp\left\{ \int_0^{x} ds \left[\overline{N} \left(1-\exp \left\{-\int_0^\zeta du f(s,s-\mathbf{e}(u)) \right\}, H<s-b \right) +\overline{N}(\zeta=\infty) \right] \right\}, \nonumber
\end{align}
where $A=\infty$ if $\Psi'(0+)\neq 0$ and is a positive constant if $\Psi'(0+)=0$. In case $\Psi'(0+)\geq 0$, the measure $\overline{N}$ assigns mass 0 to excursions with infinite lifetime and therefore one can omit the terms $\zeta<\infty$ and $\zeta=\infty$ in the definition. This definition coincides with (\ref{DefWfN}) except for a positive constant and in particular it also satisfies the exit problem in (\ref{exitWf}).
\end{proposition}
\begin{proof}[Proof of Proposition \ref{Pr1}.]
For the first exponential in (\ref{DefWf}), we begin with the case $\Psi'(0+)>0$. Denote by $\underline{X}_t$ the infimum of $X$ up to time $t$. Following the construction in \cite[p.235]{kyprianou2014fluctuations}, the scale function $W$ can be defined as $$W(a)=\frac1{\Psi'(0+)} \mathbb{P}(\underline{X}_\infty \geq -a),\ \ \ a>0.$$ In particular, since $X_t \to \infty$ a.s. as $t\to \infty$, then $\underline{X}_\infty >-\infty$ a.s. Therefore, $$\lim_{a\to \infty} W(a)= \lim_{a\to \infty} \frac1{\Psi'(0+)} \mathbb{P}(\underline{X}_\infty \geq -a) =\frac1{\Psi'(0+)} \in (0,\infty).$$ From equation (\ref{equation1}), for $0<x<a$, $$W(x)=W(a)\exp \left\{-\int_x^a ds \overline{N}(H>s) \right\}$$ and letting $a \to \infty$ we obtain $$W(x)=\frac1{\Psi'(0+)} \exp \left\{-\int_x^\infty ds \overline{N}(H>s) \right\}\in (0,\infty).$$ As a consequence, the first exponential in ($\ref{DefWf}$) coincides with $W(x-b)$ except for the constant $1/\Psi'(0+)$.

If $\Psi'(0+)=0$, for $A>0$ and $x\in (b,A+b)$ from (\ref{equation1}) $$\exp\left\{-\int_{x-b}^A ds \overline{N}\left(H>s \right) \right\}=\mathbb{P}_x \left(\tau_{A+b}^+ <\tau_b^- \right)=\frac{W(x-b)}{W(A)} \in (0,1)$$ and therefore the exponential coincides with $W(x-b)$ except for the constant $1/W(A)$. Notice that in this case we cannot let $A\to\infty$, since the term in the middle converges to $\mathbb{P}_x \left(\tau_b^- =\infty \right)$ and this probability is zero since $X$ oscillates.

Finally, if $\Psi'(0+)<0$ then $\Phi(0)>0$ and using Lemma \ref{LmPhiq} we know that $X$ diverges a.s. to $\infty$ under $\mathbb{P}^{\Phi(0)}$ and that its excursion measure is $\overline{N}\left(1_{\{\zeta<\infty \}},\cdot \right).$ In particular, from the first case, $$\exp\left\{-\int_{x-b}^\infty ds \overline{N} \left(\zeta<\infty,H>s \right) \right\}=\frac1{\Psi'(\Phi(0)+)} W^{\Phi(0)}(x-b) \in (0,\infty),$$ where $W^{\Phi(0)}$ is the scale function for $X$ under $\mathbb{P}^{\Phi(0)}$. From equation (8.16) in \cite{kyprianou2014fluctuations}, the scale function in this case can be expressed by $$W(x-b)=e^{\Phi(0)(x-b)}W^{\Phi(0)}(x-b), \ x>b.$$ Since $\Phi(0)=\overline{N}(\zeta=\infty)$ then
\begin{eqnarray*}
W_f(x,b)&=&W(x-b)\exp\left\{ \int_0^{x-b} ds \overline{N} \left(1-\exp \left\{-\int_0^\zeta du f(s+b,s+b-\mathbf{e}(u)) \right\}, H<s \right)   \right\} \\
&=& \exp \left\{(x-b) \overline{N}(\zeta=\infty) \right\} W^{\Phi(0)}(x-b) \\
&\times & \exp\left\{ \int_0^{x-b} ds \overline{N} \left(1-\exp \left\{-\int_0^\zeta du f(s+b,s+b-\mathbf{e}(u)) \right\}, H<s \right)   \right\} \\
&=& \Psi'(\Phi(0)+) \exp\left\{-\int_{x-b}^\infty ds \overline{N} \left(\zeta<\infty,H>s \right) \right\} \\
&\times & \exp\left\{ \int_0^{x-b} ds \left[\overline{N} \left(1-\exp \left\{-\int_0^\zeta du f(s+b,s+b-\mathbf{e}(u)) \right\}, H<s \right)+\overline{N}(\zeta=\infty) \right]   \right\},
\end{eqnarray*}
which proves we can define $W_f$ as in (\ref{DefWf}) except for the positive constant $\Psi'(\Phi(0)+)$.
\end{proof}

\begin{lemma} \label{Lm2}
The functions $W_f$ and $Z_{f,g,h}$ in (\ref{DefWfN}) and (\ref{DefZfg}) are well defined.
\end{lemma}

\begin{proof}
For the exponential term in (\ref{DefWfN}), since $f$ is non negative and bounded, say $0\leq f \leq K$, $K>0$, then $$0\leq 1-e^{-\int _0^\zeta f(s+b,s+b-\mathbf{e}(u))du} \leq 1-e^{-K\zeta}.$$ As we have seen in the Preliminaries, $$\overline{N}(1-e^{-K\zeta})= \Phi(K)\in(0,\infty), \ \forall K>0.$$ Therefore, in (\ref{DefWfN}) we can bound the integral in the second exponential by
\begin{eqnarray*}
\int_0^{x-b} ds \overline{N} \left(1-\exp\left\{-\int_0^\zeta du f(s+b,s+b-\mathbf{e}(u)) \right\},H<s \right) &\leq & \int_0^{x-b} ds \overline{N} \left(1-e^{-K\zeta} \right) \\
&\leq & (x-b) \Phi(K) .
\end{eqnarray*}
Thus, we conclude $W_f$ is well defined.

We now work with $Z_{f,g,h}$. From Lemma \ref{Lm1}, we have for any $b<x<z$
\begin{eqnarray*}
\frac{W_f(x,b)}{W_f(z,b)}&=& \frac{W(x-b)}{W(z-b)} \exp \left\{- \int_{x}^{z} ds  \overline{N} \left(1-\exp\left\{-\int_0^\zeta du f(s,s-\mathbf{e}(u)) \right\}, H<s-b \right)\right\}.
\end{eqnarray*}
Hence, we can express (\ref{L1eq2}) as
\begin{align} \label{eqaux1}
&\mathbb{E}_x \left[ g\left(S_{\tau_{b}^-} \right)h\left(X_{\tau_b^--},X_{\tau_b^-} \right) \exp \left\{ - \int_0^{T} dt f(S_t,X_t) \right\}, \tau_{b}^- < \tau_{a}^+ \right] \nonumber \\
& =\int_x^{a} dzg(z) \frac{W_f(x,b)}{W_f(z,b)}\kappa_b(z;f,h).
\end{align}
Taking $b<x$ fixed, the left hand side of the last equation is bounded by $||g||\cdot ||h||$, $\forall a>x$. Therefore, letting $a\to \infty$ by dominated convergence we obtain
\begin{align*}
& \mathbb{E}_x \left[ g\left(S_{\tau_{b}^-} \right) h\left(X_{\tau_b^--},X_{\tau_b^-} \right) \exp \left\{ - \int_0^{\tau_b^-} dt f(S_t,X_t) \right\}, \tau_{b}^- < \infty \right] \\
& = \lim_{a\to \infty} \mathbb{E}_x \left[ g\left(S_{\tau_{b}^-} \right) h\left(X_{\tau_b^--},X_{\tau_b^-} \right) \exp \left\{ - \int_0^{T} dt f(S_t,X_t) \right\}, \tau_{b}^- < \tau_{a}^+ \right] \\
& =  \lim_{a\to \infty} \int_x^{a} dzg(z) \frac{W_f(x,b)}{W_f(z,b)}\kappa_b(z;f,h) \\
&= \int_x^{\infty} dzg(z) \frac{W_f(x,b)}{W_f(z,b)}\kappa_b(z;f,h).
\end{align*}

In particular, (\ref{DefZfg}) is well defined. This finishes the proof.
\end{proof}
We have now all the elements to provide the proof of our first main theorem. 
\begin{proof}[Proof of Theorem \ref{Th1}]

The first identity follows directly from (\ref{DefWfN}) and Lemma \ref{Lm1}, 
\begin{eqnarray*}
\frac{W_f(x,b)}{W_f(a,b)}&=& \frac{W(x-b)}{W(a-b)} \exp \left\{- \int_{x}^{a} ds  \overline{N} \left(1-\exp\left\{-\int_0^\zeta du f(s,s-\mathbf{e}(u)) \right\}, H<s-b \right)\right\} \\
&=& \mathbb{E}_x \left[\exp \left\{ -\int_{0}^{T} dt f(S_t,X_t) \right\} , \tau_a^+ < \tau_{b}^- \right].
\end{eqnarray*}

The second identity is obtained in a similar fashion, for notice the following identities hold true
\begin{align}
& Z_{f,g,h}(x,b)-\frac{W_f(x,b)}{W_f(a,b)}Z_{f,g,h}(a,b) \nonumber \\
&= W_f (x,b) \left(1 + \int_{x}^{\infty} dz g(z) \frac{\kappa_b(z;f,h)}{ W_f (z,b)} \right)-\frac{W_f(x,b)}{W_f(a,b)} \left[ W_f(a,b) \left(1 + \int_{a}^{\infty} dz g(z) \frac{\kappa_b(z;f,h)}{ W_f (z,b)} \right) \right]\nonumber \\
&= \int_{x}^{a} dz g(z)  \frac{W_f (x,b)}{W_f (z,b)} \kappa_b(z;f,h) \label{Eqgzf}\\
&= \int_x^{a} dzg(z) \frac{W(x-b)}{W(z-b)} \exp \left\{-\int_{x}^z ds \overline{N} \left(1-\exp \left\{ -\int_0^\zeta du f(s,s-\mathbf{e}(u))\right\},H<s-b \right) \right\}\nonumber  \\
& \times \overline{N} \left(h\left(z - \mathbf{e}\left(\tilde{\tau}_{z-b}- \right),z - \mathbf{e}\left(\tilde{\tau}_{z-b} \right) \right) \exp \left\{-\int_0^{\tilde{\tau}^+_{z-b}} du f(z,z-\mathbf{e}(u)) \right\} , H>z-b\right)\nonumber  \\
&= \mathbb{E}_x \left[ g\left(S_{\tau_{-b}^-} \right) h\left(X_{\tau_b^--},X_{\tau_b^-} \right) \exp \left\{ - \int_0^{T} dt f(S_t,X_t)  \right\}, \tau_{b}^- < \tau_{a}^+ \right]. \nonumber 
\end{align}We have so finished the proof of Theorem~\ref{Th2}.
\end{proof}

\begin{proof}[Proof of Theorem \ref{Th2}]

Since $X$ does not have positive jumps and has unbounded variation, then $S_{T}=a$ if and only if it exits the interval by above, that is, on the event $\{\tau_a^+ < \tau_b^-\}$. Using (\ref{L1eq1}) and the fact that $\mathbb{P}_x(\tau_a^+<\tau_b^-)=W(x-b)/W(a-b)$, we obtain 
\begin{align*}
\mathbb{E}_x & \left[\exp\left\{-\int_0^{T} dt f(S_t,X_t) \right\} \bigg| S_{T}=a \right] \\
&= \exp \left\{-\int_x^a ds \overline{N} \left(1-\exp \left\{ -\int_0^\zeta duf(s,s-\mathbf{e}(u))\right\}, H<s-b \right) \right\}.
\end{align*}
Now, for $x\leq z < a$, we know that $S_{T}=z$ implies the process exits the interval by below. From (\ref{Eqgzf}) with $h\equiv 1$, we know that
\begin{align*}
\mathbb{E}_x & \left[ g\left(S_{\tau_b^-} \right) \exp \left\{ - \int_0^{T} dt f(S_t,X_t) \right\}, \tau_{b}^-<\tau_a^+  \right]\\
&= \int_x^{a} dzg(z) \frac{W_f(x,b)}{W_f(z,b)} \overline{N} \left(\exp \left\{-\int_0^{\tilde{\tau}^+_{z-b}} du f(z,z-\mathbf{e}(u)) \right\} , H>z-b\right)
\end{align*}
and therefore
\begin{eqnarray*}
& \ & \mathbb{E}_x \left[ g\left(S_{\tau_b^-} \right) \exp \left\{ - \int_0^{T} dt f(S_t,X_t) \right\} \bigg|  \tau_{b}^- <\tau_a^+  \right] \nonumber \\
&=& \int_x^{a} dzg(z) \frac{W_f(x,b)}{\mathbb{P}_x \left( \tau_{b}^- < \tau_a^+   \right)W_f(z,b)} \overline{N} \left(\exp \left\{-\int_0^{\tilde{\tau}^+_{z-b}} du f(z,z-\mathbf{e}(u)) \right\} , H>z-b\right).
\end{eqnarray*}
Let $\mu_x^{a,b}$ be the conditional law given $\{ \tau_{b}^-< \tau_a^+ \}$. From the last computation we have 
\begin{eqnarray} \label{muxab}
& \ & \mu_x^{a,b} \left( g\left(S_{\tau_{b}^-} \right) \exp \left\{ - \int_0^{T} dt f(S_t,X_t) \right\} \right) \nonumber \\
&=& \int_x^{a} dzg(z) \frac{W_f(x,b)}{\mathbb{P}_x \left(  \tau_{b}^- < \tau_a^+ \right)W_f(z,b)} \overline{N} \left(\exp \left\{-\int_0^{\tilde{\tau}^+_{z-b}} du f(z,z-\mathbf{e}(u)) \right\} , H>z-b\right).
\end{eqnarray}
Taking $f\equiv 0$ we get $W_f(\cdot,b)=W(\cdot-b)$ and since the equation is valid for every measurable and bounded $g$, we can conclude that the law of $S_{\tau_{b}^-}$ under $\mu_x^{a,b}$ has density $$\nu(z):=\nu_x^{a,b}(z)=\frac1{\mathbb{P}_x \left(  \tau_{b}^- < \tau_a^+ \right)} \frac{W(x-b)}{W(z-b)} \overline{N}(H>z-b).$$
On the other side, taking conditional expectation given $S_{ \tau_{b}^-}$ we have
\begin{eqnarray*}
\mu_x^{a,b} \left( g\left(S_{ \tau_{b}^-} \right) \exp \left\{ - \int_0^{T} dt f(S_t,X_t)  \right\} \right)&=&\mu_x^{a,b} \left( g\left(S_{\tau_{b}^-} \right) \mu_x^{a,b} \left(  \exp \left\{ - \int_0^{T} dt f(S_t,X_t) \right\}  \bigg| S_{ \tau_{b}^-} \right) \right) \\
&=& \mu_x^{a,b} \left( g\left(S_{ \tau_{b}^-} \right) K\left(S_{ \tau_{b}^-} \right) \right),
\end{eqnarray*}
for a certain function $K$ depending on $f$. As we have seen before, the latter can be written as $$\mu_x^{a,b} \left( g\left(S_{ \tau_{b}^-} \right) K \left(S_{ \tau_{b}^-} \right) \right)= \int_x^a dz g(z) K(z) \nu(z)$$ Again, since this is valid for every $g$ measurable and bounded comparing with (\ref{muxab}) we conclude 
\begin{eqnarray*}
K(z)&=&\frac{W(z-b)}{W(x-b)} \frac{W_f(x,b)}{W_f(z,b)} \frac{\overline{N} \left(\exp \left\{-\int_0^{\tilde{\tau}^+_{z-b}}du f(z,z-\mathbf{e}(u)) \right\} , H>z-b\right)}{\overline{N}(H>z-b)} .
\end{eqnarray*}
In a similar way, there is a function $\tilde{K}$ such that $$\mathbb{E}_x \left[\exp \left\{ - \int_0^{T} dt f(S_t,X_t) \right\} \bigg| S_{\tau_{b}^-}, \tau_{b}^-< \tau_a^+   \right]=\tilde{K}\left(S_{ \tau_{b}^-} \right).$$ This follows since the event $\{\tau_{b}^- < \tau_a^+ \}$ coincides almost surely with $\{S_{T} \in [x,a) \}$ because of the absence of positive jumps. Hence,

\begin{eqnarray*}
\int_{x}^a dz g(z) K(z) \nu(z) &=& \mu_x^{a,b} \left( g\left(S_{ \tau_{b}^-} \right) \exp \left\{ - \int_0^{T} dt f(S_t,X_t)  \right\} \right) \\
&=& \frac{\mathbb{E}_x \left[ g\left(S_{\tau_{b}^-} \right) \exp \left\{ - \int_0^{T} dt f(S_t,X_t)  \right\} ,  \tau_{b}^- < \tau_a^+ \right]} {\mathbb{E}_x \left[1_{\{ \tau_{b}^- < \tau_a^+ \}} \right]} \\
&=& \frac{\mathbb{E}_x \left[ g\left(S_{\tau_{b}^-} \right) 1_{\{\tau_{b}^- < \tau_a^+   \}} \mathbb{E}_x \left[ \exp \left\{ - \int_0^{T} dt f(S_t,X_t) \right\} \bigg| S_{ \tau_{b}^-}, \tau_{b}^- < \tau_a^+   \right] \right]} {\mathbb{E}_x \left[1_{\{ \tau_{b}^- < \tau_a^+ \}} \right]} \\
&=& \frac{\mathbb{E}_x \left[ g\left(S_{\tau_{b}^-} \right) 1_{\{ \tau_{b}^- < \tau_a^+ \}} \tilde{K} \left(S_{ \tau_{b}^-} \right) \right]} {\mathbb{E}_x \left[1_{\{ \tau_{b}^- <\tau_a^+  \}} \right]}\\
&=& \mu_x^{a,b} \left(g\left(S_{ \tau_{b}^-} \right) \tilde{K}\left(S_{ \tau_{b}^-} \right) \right) \\
&=& \int_{x}^a dz g(z) \tilde{K}(z) \nu(z).
\end{eqnarray*}

Since the latter is valid for all $g$ measurable and bounded, it implies that in the case the process exits by below, the functional we were looking for coincides with the function $K$ a.s. In other words, for $z\in [x,a)$,
\begin{align*}
\mathbb{E}_x &\left[\exp\left\{-\int_0^{T} dt f(S_t,X_t) \right\} \bigg|  S_{T}=z \right] \\
& = \frac{W(x-b)}{W(z-b)} \frac{W_f(x,b)}{W_f(z,b)} \frac{\overline{N} \left(\exp \left\{-\int_0^{\tilde{\tau}^+_{z-b}}du f(z,z-\mathbf{e}(u)) \right\} , H>z-b\right)}{\overline{N}(H>z-b)}
\end{align*}
This concludes the proof.
\end{proof}

\begin{proof}[Proof of Theorem \ref{Th3}]

Take $b\in \mathbb{R}$ fixed. We already know that $\forall a>x>b$, $$\frac{W_f(x,b)}{W_f(a,b)}=\frac{W^{(f)}(x,b)}{W^{(f)}(a,b)}$$ and therefore, $$W^{(f)}(x,b)=\frac{W^{(f)}(a,b)}{W_f(a,b)} W_f(x,b).$$ Since this holds for every $a>x$ and $W_f(x,b)>0$, the quotient $\frac{W^{(f)}(a,b)}{W_f(a,b)}$ is constant in $[x,\infty)$, say $$\frac{W^{(f)}(a,b)}{W_f(a,b)}=\gamma_x.$$ Moreover, if we take $x_1<x_2<a$ such that $b<x_1<x_2<a$, using the previous argument we have $\frac{W^{(f)}(a,b)}{W_f(a,b)}=\gamma_{x_1}$ and also $\frac{W^{(f)}(a,b)}{W_f(a,b)}=\gamma_{x_2}$. Therefore, this constant does not depend on $x$ and thus we can write $$W^{(f)}(x,b)=\gamma W_f(x,b), \ \ \ x>b,$$ that is, $W^{(f)}(\cdot,b)$ and $W_f(\cdot,b)$ differ by a constant.

A similar argument can be used for $Z_f$. With $b$ fixed, from Theorem \ref{Th1} we have for every $a>x>b$ that $$Z^{(f)}(x,b)-\frac{W^{(f)}(x,b)}{W^{(f)}(a,b)}Z^{(f)}(a,b)=Z_f(x,b)-\frac{W_f(x,b)}{W_f(a,b)}Z_f(a,b).$$ Since we already know $W^{(f)}(x,b)/W^{(f)}(a,b)$ equals $W_f(x,b)/W_f(a,b)$, we can rewrite the last expression as $$\frac{Z^{(f)}(x,b)-Z_f(x,b)}{W_f(x,b)} = \frac{Z^{(f)}(a,b)-Z_f(a,b)}{W_f(a,b)}.$$ The latter implies that the quotient $\frac{Z^{(f)}(a,b)-Z_f(a,b)}{W_f(a,b)}$ is constant in $[x,\infty)$ for any $x>b$ and as before, the constant does not depend on $x$. Therefore, $$\frac{Z^{(f)}(x,b)-Z_f(x,b)}{W_f(x,b)}=\alpha, \ \ \ \forall x \in (b,\infty)$$ and we obtain
\begin{eqnarray*}
Z^{(f)}(x,b)&=&\alpha W_f(x,b) + Z_f(x,b) \\
&=& W_f(x,b) \left(1 + \alpha +\int_x^\infty dz \frac{\overline{N} \left(\exp \left\{-\int_0^{\tilde{\tau}_{z-b}^+} du f(z,z-\mathbf{e}(u)) \right\},H>z-b \right)}{W_f(z,b)} \right).
\end{eqnarray*}
Hence, we can change the number 1 in the original definition of $Z_f$ in a proper way so $Z^{(f)}(\cdot,b)$ and $Z_f(\cdot,b)$ coincide.
\end{proof}

\bibliographystyle{plain}
\bibliography{scalef}

\end{document}